\documentclass{birkau}

\usepackage{amssymb, latexsym, url}

\newcommand{\Om}{\Omega}

\newcommand{\mbZ}{\mathbb{Z}}
\newcommand{\mbN}{\mathbb{N}}

\newcommand{\mc}{\mathcal}
\newcommand{\mcV}{\mathcal{V}}
\newcommand{\mcK}{\mathcal{K}}

\newcommand{\mcR}{\mathcal{R}}
\newcommand{\mcQ}{\mathcal{Q}}
\newcommand{\mcT}{\mathcal{T}}
\newcommand{\mcW}{\mathcal{W}}
\newcommand{\mcS}{\mathcal{S}}

\newcommand{\mcU}{\mathcal{U}}

\newcommand{\wt}{\widetilde}

\newcommand{\mf}{\mathsf}
\newcommand{\mr}{\mathrm}

\newcommand{\bfx}{\mathbf x}

\newcommand{\bfa}{\mathbf a}

\DeclareMathOperator{\tr}{tr}

\theoremstyle{plain}
\newtheorem{theorem}{Theorem}[section]
\newtheorem{corollary}[theorem]{Corollary}
\newtheorem{proposition}[theorem]{Proposition}
\newtheorem{lemma}[theorem]{Lemma}

\theoremstyle{definition}
\newtheorem{definition}[theorem]{Definition}

\newtheorem{example}[theorem]{Example}

\numberwithin{equation}{section}

\input xy
\xyoption{all}

\begin{document}

\title[Mal'tsev products]{Mal'tsev products of varieties, I}

\author[T. Penza]{T. Penza}
\address{Faculty of Mathematics and Information Science\\
Warsaw University of Technology\\
00-661 Warsaw, Poland}
\email{T.Penza@mini.pw.edu.pl}

\author[A.B. Romanowska]{A.B. Romanowska}
\address{Faculty of Mathematics and Information Science\\
Warsaw University of Technology\\
00-661 Warsaw, Poland}
\email{A.Romanowska@mini.pw.edu.pl}

\thanks{The first author's research was supported by the Warsaw University of Technology under grant number 504/04259/1120.}



\begin{abstract}
We investigate the Mal'tsev product $\mcV \circ \mcW$ of two varieties $\mcV$ and  $\mcW$ of the same similarity type. Such a product is usually a quasivariety but not necessarily a variety. We give an equational base for the variety generated by $\mcV \circ \mcW$ in terms of identities satisfied in $\mcV$ and $\mcW$. Then the main result provides a new sufficient condition for $\mcV \circ \mcW$ to be a variety: If $\mcW$ is an idempotent variety and there are terms $f(x,y)$ and $g(x,y)$ such that $\mcW$ satisfies the identity $f(x,y) = g(x,y)$ and $\mcV$ satisfies the identities $f(x,y) = x$ and $g(x,y) = y$, then $\mcV \circ \mcW$ is a variety.
We also provide a number of examples and applications of this result.
\end{abstract}

\maketitle

\section{Introduction}

In his 1967 paper \cite{M67}, Mal'tsev introduced a product of classes of algebras that is now known as the Mal'tsev product.
Let $\tau\colon \Omega \rightarrow \mathbb{N}$ be a similarity type of
$\Omega$-algebras.
Let $\mc{B}$ and $\mc{C}$ be subclasses of a class $\mc{D}$ of $\Om$-algebras. Then the  \emph{Mal'tsev product $\mc{B} \circ_\mc{D} \mc{C}$ of $\mc{B}$ and $\mc{C}$ relative to $\mc{D}$} consists of all algebras $A$ in $\mc{D}$, with a congruence $\theta$, such that $A/\theta$ belongs to $\mc{C}$ and every congruence class of $\theta$ that is a subalgebra of $A$ belongs to $\mc{B}$. If $\mc{D}$ is the variety $\mcT$ of all $\Om$-algebras, then the Mal'tsev product $\mc{B} \circ_{\mc{T}} \mc{C}$ is called simply the \emph{Mal'tsev product} of $\mc{B}$ and $\mc{C}$, and is denoted by $\mc{B} \circ \mc{C}$.

In the same paper, Mal'tsev proved that, if $\mcQ$ and $\mcR$ are subquasivarieties of a quasivariety $\mcK$ of $\Om$-algebras, then the Mal'tsev product $\mc{Q} \circ_{\mcK} \mc{R}$ is also a quasivariety, albeit under the assumption of a finite similarity type for $\mcK$. If one assumes that the second quasivariety $\mc{R}$ is idempotent, then there is no restriction on the type.
This ``closure'' property makes the class of subquasivarieties of a given quasivariety a natural domain for consideration of Mal'tsev products.
Moreover,
the congruence $\theta$ in the definition of Mal'tsev product of two quasivarieties may be taken to be the $\mcR$-replica congruence $\varrho$ of $A$, the smallest congruence of $A$ whose induced quotient falls into the quasivariety $\mcR$. (See e.g. \cite[Ch.~3]{RS02}.) So in this case, the Mal'tsev product may be described formally as follows:
\begin{equation}
\mc{Q} \circ_{\mcK} \mc{R} = \{A \in \mcK \mid (\forall{a \in A}) \,\, (a/\varrho \leq A \, \Rightarrow \, a/\varrho \in \mcQ)\}.
\end{equation}

Idempotent elements of $\mcR$-algebras (i.e. algebras in $\mcR$) play a significant role in the theory of Mal'tsev products $\mc{Q} \circ_{\mcK} \mc{R}$ of quasivarieties $\mcQ$ and $\mcR$.
Recall that an element $a$ of an $\Om$-algebra $A$ is \emph{idempotent} if $\{a\}$ is a subalgebra of $A$. An algebra $A$ is idempotent if every element is idempotent. A class of $\Om$-algebras is idempotent if every member algebra is idempotent.

The following easy lemma shows a general relation between subalgebras of an algebra and idempotents of its quotients.

\begin{lemma}\label{L:subcon}
Let $A$ be an $\Om$-algebra with congruence $\theta$, and let $a \in A$. Then
$a/\theta$ is a subalgebra of $A$ if and only if $a/\theta$ is an idempotent element of $A/\theta$.
\end{lemma}
\begin{proof}
If $a/\theta$ is a subalgebra of $A$, then for $a_1, \dots, a_n \in a/\theta$ and each ($n$-ary)
$\omega \in \Om$, we have $a_1 \dots a_n \omega \, \in \, a/\theta$. Hence $a/\theta \dots a/\theta \, \omega = a\dots a \omega /\theta = a/\theta$.

Now let $a/\theta$ be an idempotent element of $A/\theta$. This means that $a/\theta \dots a/\theta \, \omega = a/\theta$. Hence for $a_1, \dots, a_n \in a/\theta$ and each $n$-ary $\omega \in \Om$, we have $a_1 \dots a_n \omega /\theta = a_1/\theta \dots a_n/\theta \, \omega = a/\theta \dots a/\theta \, \omega = a/\theta$. Consequently $a/\theta$ is a subalgebra of $A$.
\end{proof}

Let us note that if no non-trivial $\mcR$-algebra has idempotent elements, then $\mc{Q} \circ_{\mcK} \mc{R}$ contains all $\mcK$-algebras, since each of them has the $\mcR$-replica congruence and no corresponding congruence classes are subalgebras. So in this case $\mc{Q} \circ_{\mcK} \mc{R}$ coincides with the class $\mcK$.

Another extreme case concerns the situation when $\mcR$ is an idempotent quasivariety. In this case, each
$\varrho$-class of an algebra $A$ in $\mc{Q} \circ_{\mcK} \mc{R}$ is a subalgebra of $A$, and then
\begin{equation}
\mc{Q} \circ_{\mcK} \mc{R} = \{A \in \mcK \mid  (\forall{a \in A}) \,\,(a/\varrho \in \mcQ)\}.
\end{equation}

Finally it may happen that, though $\mcR$ is not an idempotent quasivariety, all $\mcR$-algebras have some idempotent elements. (Examples are provided by inverse semigroups.)

We usually assume that $\Om$-algebras considered in this paper
have no nullary basic operations. Note however that this assumption is not essential. If $\Om$ contains a symbol $c$ of a nullary operation, one can replace it by a symbol $c(x)$ of a unary operation as follows. If $\mcQ$ is a quasivariety of $\Om$-algebras of a type $\tau$ containing the symbol $c$, then instead of $\mcQ$ one can consider the equivalent quasivariety $\mcQ'$ of the type obtained from $\tau$ by taking the symbol $c(x)$ in place of $c$, satisfying the quasi-identities true in $\mcQ$ and additionally the identity $c(x) = c(y)$. Hence, in particular, the unary operation is constant on all algebras of $\mcQ'$.
Sometimes we have to consider varieties of algebras with at least one non-unary operation and with no nullary basic operations. Such a similarity type is called \emph{plural} \cite{RS02}.

An $\Om$-term $t$ such that $\mcV$ satisfies the identities
\begin{equation}\label{E:tid}
t \dots t\,\omega = t,
\end{equation}
for each $\omega \in \Om$, will be called a \emph{term idempotent} of $\mcV$.

For varieties of algebras the following version of \cite[Thm.~~9]{C75} holds.

\begin{lemma}\label{L:termidem}
Let $\mcV$ be a variety of $\Om$-algebras of a type without nullary operation symbols. Then each algebra in $\mcV$ has idempotents if and only if the variety $\mcV$ has a term idempotent. Moreover, if $\mcV$ is of a plural type and has a term idempotent, then it has term idempotents of any arity $n \geq 1$.
\end{lemma}
\begin{proof}
If all algebras in $\mcV$ have idempotents, then in particular also the free algebra  $X\Om$ on one-element set $X$ has an idempotent, which is some unary term $t$. And as an idempotent, this term satisfies \eqref{E:tid}.

On the other hand, if there is an $n$-ary $\Om$-term $t$ such that $\mcV$ satisfies \eqref{E:tid}, then for any algebra $A$ in $\mcV$ and any $a \in A$, one has
\[
(a \dots a t) \dots (a \dots a t) \omega = a \dots a t
\]
for each $\omega \in \Om$, whence $a \dots a t$ is an idempotent of $A$.

Finally note that if $x_1 \dots x_n t$ is an $\Om$-term satisfying \eqref{E:tid} for each $\omega \in \Om$, then the unary $x \dots x t$ also satisfies it. And moreover for arbitrary $\Om$-term $u$ of any arity, we also have
\[
(u \dots u t) \dots (u\dots u t) \, \omega = u \dots u t
\]
for each $\omega \in \Om$.
\end{proof}

A special case concerns the situation of so-called polarized varieties.
A variety $\mcV$ of $\Om$-algebras is \emph{polarized} if there is a unary term idempotent of $\mcV$ that is constant on every member of $\mcV$.
This constant is called the \emph{pole} of the algebra, and the algebra is called \emph{polarized}.
The pole of an algebra is unique (if it exists), and a congruence class of a polarized algebra is a subalgebra if and only if it is the congruence class of the pole. A (quasi)variety of $\Om$-algebras is \emph{polarized} if all its members are polarized algebras. (See Mal'tsev~\cite{M67}.)

\begin{corollary}\label{C:idempolar}
Each algebra in $\mcV$ has precisely one idempotent if and only if the variety $\mcV$ is polarized.
\end{corollary}

In this paper we are interested in Mal'tsev products $\mcV \circ \mcW$ of two varieties $\mcV$ and $\mcW$ of the same type $\tau \colon \Om \rightarrow \mbN$. Such a product is usually a quasivariety, however need not be a variety.
For some results providing sufficient conditions for the Mal'tsev product of two varieties to be a variety see e.g. \cite[S.~57,~63]{G79}, \cite{N67}, \cite{M67}, \cite{I84a} and \cite{I84b}.
They usually concern Mal'tsev products relative to some special varieties. The best known may be the following theorem of Mal'tsev.

\begin{theorem}\cite{M67}\label{T:MT}
If $\mcV$ and $\mcW$ are subvarieties of a congruence permutable polarized variety $\mcK$, then $\mcV \circ_{\mcK} \mcW$ is a variety.
\end{theorem}

Such products were considered earlier in the case of groups (see Hanna Neumann's book \cite{N67} on varieties of groups).
The Mal'tsev product of any two varieties of groups relative to the variety of all groups is a variety. Sufficient conditions of a kind similar to Theorem~\ref{T:MT} were then obtained by A.A.~Iskander in \cite{I84a} and \cite{I84b}, and some versions of Theorem~\ref{T:MT} by C. Bergman in \cite{B20} and \cite{B18}.

In particular, almost the same proof as the proof of Mal'tsev's Theorem~\ref{T:MT} may be used to show the following.

\begin{theorem}\cite{B18}\label{T:BT}
If $\mcV$ and $\mcW$ are idempotent subvarieties of a congruence permutable variety $\mcK$, then $\mcV \circ_{\mcK} \mcW$ is a variety.
\end{theorem}

In fact, it is sufficient to assume that only the variety $\mcW$ is idempotent.

Another special case concerns the Mal'tsev product of a strongly irregular variety $\mcV$ of any plural type $\tau$ and the variety of semilattices considered as the variety of the same type $\tau$. Recall that a strongly irregular variety is a variety $\mcV_t$ defined by a set of regular identities and a (strongly irregular) identity $t(x,y) = x$, where $t(x,y)$ is a binary term containing both variables $x$ and $y$. The variety $\mcS$ of semilattices may be considered as the variety $\mcS_{\tau}$ of any plural type $\tau$ (definitionally) equivalent to $\mcS$. (See \cite{BPR19} for details.)

\begin{theorem}\cite{BPR19}\label{T:BPR}
If $\mcV_t$ is a strongly irregular variety of a plural type $\tau$ and $\mcS_{\tau}$ is the variety of the same type $\tau$ equivalent to the variety of semilattices, then $\mcV_t \circ \mcS_{\tau}$ is a variety.
\end{theorem}

Also in \cite{BPR19}, it was shown how to derive an equational base for the product $\mcV_t \circ \mcS_{\tau}$ from equational bases for $\mcV_t$ and $\mcS_{\tau}$, and even more generally, how to derive all identities true in the product $\mcV \circ \mcS_{\tau}$, for any variety $\mcV$ of $\Om$-algebras of a plural type $\tau$.

The aim of this paper is to describe the identities true in the Mal'tsev product $\mcV \circ \mcW$ of two varieties of a type $\tau$ without symbols of nullary operations in terms of the identities satisfied by its components, and then to extend Theorem~\ref{T:BPR}.

The identities satisfied in $\mcV \circ \mcW$ are described in Section~\ref{S:Ids}. They are obtained from the identities true in $\mcV$ by
substituting for their variables $\mcW$-equivalent term idempotents of $\mcW$. In the case when the Mal'tsev product $\mcV \circ \mcW$ is a variety, one then obtains its  equational base.

Note that if $\mcU$ is a variety of type $\tau$ and $\mcV \circ \mcW$ is a variety, then $\mcV \circ_{\mcU} \mcW = (\mcV \circ \mcW) \cap \mcU$ is also a variety.
However, it is possible that $\mcV \circ_{\mcU} \mcW $ is a variety but $\mcV \circ \mcW$ is not. It was shown in \cite{BPR19}, that the Mal'tsev product $\mcS \circ \mcS$, where $\mcS$ is the variety of semilattices, is not a variety. But since $\mcS \subseteq \mcS \circ \mcS$, it follows that $\mcS \circ_{\mcS} \mcS$ is the variety $\mcS$.

Theorem~\ref{T:main} of Section~\ref{S:Main} provides a new sufficient condition for $\mcV \circ \mcW$ to be a variety.
We show that if $\mcW$ is idempotent and there are terms $f(x,y)$ and $g(x,y)$ such that $\mcV$ satisfies the identities $f(x,y) = x$ and $g(x,y) = y$, and $\mcW$ satisfies the identity $f(x,y) = g(x,y)$, then $\mcV \circ \mcW$ is a variety. Some consequences of this theorem are derived. Section~\ref{S:Bands}
discusses the case when the variety $\mcW$ is equivalent to a variety of bands.

We use notation and conventions similar to those of \cite{BPR19}, \cite{RS99, RS02}. For details and
further information concerning quasivarieties and Mal'tsev product of quasivarieties we
refer the reader to \cite{M67} and \cite{M71}, and then also to \cite{B18} and
\cite[Ch.~2]{RS02}; for universal algebra, see \cite{CB12} and \cite{RS99}, \cite{RS02}.

\section{Identities true in Mal'tsev products}\label{S:Ids}

In the whole paper we assume that $\mcV$ and $\mcW$ are varieties of the same type $\tau\colon \Omega \rightarrow \mbN$ without symbols of nullary operations.
Let $\mr{Id}(\mcV)$ and $\mr{Id}(\mcW)$ denote the sets of identities true in $\mcV$ and $\mcW$, respectively.
Members of the absolutely free $\Omega$-algebra $X\Omega$ over a countably infinite set $X$, i.e. $\Om$-terms, are denoted $x_1 \dots x_k u, x_1 \dots x_k v$, etc. or briefly $\bfx u, \bfx v$, etc. or just $u, v$, etc. Note that $\bfx u$ may not actually contain a particular variable among $x_1, \dots, x_k$. We use a similar notation for elements of an $\Omega$-algebra $A$. If $a_1, \dots, a_k$ are members of $A$, we write $\bfa u$, when $u$ is applied to elements $a_1, \dots, a_k$, even if a particular variable $x_i$ does not appear in $u$.
If the identity $u =  v$ holds in a variety $\mcW$, then the terms $u$ and $v$ will sometimes be called $\mcW$-\emph{equivalent}.

\begin{definition}\label{D:eqbasprol}
Let
\begin{equation}\label{E:s}
x_1 \dots x_k u = x_1 \dots x_k v \tag{$\sigma$}
\end{equation}
be an identity true in $\mcV$.
Define
\begin{align*}
\sigma^{p} := &\{r_1 \dots r_k u = r_1 \dots r_k v \,\mid \\
& r_1, \dots, r_k \in X\Om , \\
&\forall \, 1 \leq i, j \leq k, \, \mcW \models r_i = r_j , \\
&\forall \, \omega \in \Om, \, \mcW \models r_1 \dots r_1 \omega = r_1\}.
\end{align*}
For a subset $\Sigma$ of $\mr{Id}(\mcV)$, define
\[
\Sigma^p := \bigcup_{\sigma \in \Sigma} \sigma^p.
\]
\end{definition}

Note that the last condition of the definition of $\sigma^p$ implies that
$\mcW \models r_i \dots r_i \omega = r_i$ for all $1 \leq i\leq k$ and $\omega \in \Omega$. If $\mcW$ is an idempotent variety, then this condition is always satisfied.

In other words, we may say that the identities $\sigma^p$ are obtained from $\sigma$ by substituting for variables of $\sigma$, pairwise $\mcW$-equivalent term idempotents of $\mcW$.
However, by Lemma~\ref{L:termidem}, such terms can exist only if all $\mcW$-algebras contain idempotent elements. Otherwise the set $\Sigma^p$ is empty.

\begin{lemma}\label{L:eqbase}
Let $\Sigma$ be any set of identities true in $\mcV$. Then the Mal'tsev product $\mcV \circ \mcW$ satisfies the identities $\Sigma^p$.
\end{lemma}
\begin{proof}
First note that if $\mcW$ has no term idempotents, that means there are no terms $r$ such that $\mcW$ satisfies $r \dots r \omega = r$ for all $\omega \in \Om$, then the set $\Sigma^p$ is empty, and $\mcV \circ\mcW$ satisfies $\Sigma^p$ vacuously.

Now we assume that there exist such terms.
Let $A$ be in $\mcV \circ \mcW$ with the $\mcW$-replica congruence $\varrho$. Let $x_1 \dots x_k u = x_1 \dots x_k v$ be an identity of $\Sigma$, and let $y_1 \dots y_n r_1, \, \dots, \, y_1 \dots y_n r_k$
be pairwise $\mcW$-equivalent term idempotents of $\mcW$.
Since $A/\varrho \in \mcW$, it follows that for any $\bfa \in A^n$, we have
\begin{equation}\label{E:eqroclass}
(\bfa r_i)/\varrho = (\bfa r_j)/\varrho
\end{equation}
for every $1 \leq i < j \leq k$.
Thus all of the elements $\bfa r_i$, for $i = 1, \dots ,k$, lie in the same congruence class of $\varrho$.

Since all $r_i$ are term idempotents of $\mcW$, it follows that
\[
(\bfa r_{1})/\varrho \dots (\bfa r_{1})/\varrho \,\, \omega =(\bfa r_1)/\varrho
\]
for each $\omega \in \Om$, that means $(\bfa r_1)/\varrho$ is an idempotent of $A/\varrho$. Then by Lemma~\ref{L:subcon}, $(\bfa r_1)/\varrho$
is a subalgebra of $A$.
Consequently, since $A \in \mcV \circ \mcW$, we have $(\bfa r_1)/\varrho \in \mcV$.
Since $\bfa r_1, \dots, \bfa r_k \in (\bfa r_1)/\varrho$, it follows that
$\bfa r_1 \dots \bfa r_k u = \bfa r_1 \dots \bfa r_k v$. Thus the identity
\[
r_1 \dots r_k u = r_1 \dots r_k v
\]
is satisfied in $A$.
\end{proof}

Let us note that Lemma~\ref{L:eqbase} holds for any equational base $\Sigma$ of $\mcV$. In particular, it holds for $\Sigma = \mr{Id}(\mcV)$.

\begin{corollary}\label{C:VWU}
For any equational base $\Sigma$ of $\mcV$, the Mal'tsev product $\mcV \circ \mcW$ is contained in the variety $\mcU$ defined by $\Sigma^p$:
\begin{equation}
\mcV \circ \mcW \subseteq \mcU.
\end{equation}
\end{corollary}

The following simple examples illustrate how the set $\Sigma^p$ is built from the set $\Sigma$.

\begin{example}\label{Ex:lzsl}
Let $\mcV$ be the variety $\mc{LZ}$ of left-zero semigroups defined by the identity $x \cdot y = x$, and let $\mcW$ be the variety $\mcS$ of semilattices. It is known that $r = s$ is an identity true in $\mcS$ precisely if the terms $r$ and $s$ have the same sets of variables.
As the set $\Sigma$ we take the one-element set $\{\sigma\}$, where $\sigma$ is the single identity defining $\mc{LZ}$. Since semilattices are idempotent, the set $\Sigma^p$ consists of all identities of the form
\[
x_1 \dots x_m\, r \cdot x_1 \dots x_m\, s = x_1 \dots x_m\, r,
\]
where both $r$ and $s$ contain precisely the same variables $x_1, \dots, x_m$ and $m = 1, 2, \dots$. Note that $\Sigma^p$ contains the idempotent law $x \cdot x = x$, but neither the commutative law nor the associative law are contained in $\Sigma^p$. By Lemma~\ref{L:eqbase}, the Mal'tsev product $\mc{LZ} \circ \mcS$ satisfies the identities $\Sigma^p$. It follows by results of \cite{BPR19} that in fact these identities define $\mc{LZ} \circ \mcS$.
\end{example}

\begin{example}\label{Ex:lzcsgr}
Let $\mcV$ be the variety $\mc{LZ}$ and let $\mcW$ be the variety $\mc{CS}$ of constant semigroups defined by the identity $x \cdot y = z \cdot t$. It is easy to see that $r = s$ is a (non-trivial) identity true in $\mc{CS}$ precisely if none of the terms $r$ and $s$ is a variable, or in different words, if each is the product of two groupoid terms, i.e.
\[
r = r_1 \cdot r_2 \,\, \mbox{and} \,\, s = s_1 \cdot s_2.
\]
As the set $\Sigma$ we take the one-element set $\{\sigma\}$, where $\sigma$ is the identity defining $\mc{LZ}$. To describe $\Sigma^p$, note that
for $r = r_1 \cdot r_2$,
\[
r \cdot r = r_1 \cdot r_2 = r.
\]
Consequently the last condition of the definition of $\sigma^p$ is satisfied. It follows that the set $\Sigma^p$ consists of the identities of the form
\[
(r_1 \cdot r_2) \cdot (s_1 \cdot s_2) = r_1 \cdot r_2,
\]
where $r_1, r_2, s_1, s_2$ are any groupoid terms. If we define $c := x \cdot x$, then $\Sigma^p$ may be written as follows.
\[
\Sigma^p = \{r \cdot s = r \mid \mc{CS} \models r = s = c\}.
\]
Note that $\Sigma^p$ contains neither the idempotent nor the associative, nor the commutative law. By Lemma~\ref{L:eqbase}, the Mal'tsev product $\mc{LZ} \circ \mc{CS}$ satisfies the identities $\Sigma^p$.
\end{example}

\begin{example}\label{Ex:lzgr}
In this example $\mcV$ will be again the variety $\mc{LZ}$ with $\Sigma$ and $\sigma$ defined as in the previous examples. And $\mcW$ will be the variety $\mc{GP}$ of groups, considered as algebras $(G, \cdot, ^{-1})$ with one binary and one unary operation, where the group identity $e$ is defined as $x \cdot x^{-1}$. (See e.g. \cite{BPR19}.) Left-zero bands are considered as algebras of the same type, with $x^{-1} := x$. Let $r$ and $s$ be groupoid terms. Then $r = s$ is an identity satisfied in $\mc{GP}$ if $r$ and $s$ have the same reduced (or canonical) form. (See e.g. \cite{H74}.) Recall that each group has only one idempotent, namely the identity $e$. Hence $\sigma^p$ consists of the identities of the form
\[
r \cdot s = r,
\]
where $r = s$ are the identities satisfied in $\mc{GP}$ such that $r \cdot r = r$ and $r^{-1} = r$. However, the last two conditions are satisfied precisely if $r = e$. It follows that
\[
\Sigma^p = \{r \cdot s = r \mid \mc{GP} \models r = s = e \}.
\]
By Lemma~\ref{L:eqbase}, the Mal'tsev product $\mc{LZ} \circ \mc{GP}$ satisfies the identities $\Sigma^p$.
\end{example}

In the next example, algebras in $\mcW$ have more than one idempotent.

\begin{example}
In this example $\mcV$, $\sigma$ and $\Sigma$ are defined as in Example~\ref{Ex:lzgr}. The variety $\mcW$ is the regularisation $\wt{\mc{GP}}$ of the variety $\mc{GP}$ of groups, that means it is defined by the regular identities true in $\mc{GP}$. (See e.g. \cite[ch.~~4]{RS02} and \cite{BPR19}.) This class consists precisely of the P\l onka sums of groups (in other words: strong semilattices of groups or Clifford semigroups), and forms a subvariety of the variety of inverse semigroups. (We refer the reader to \cite{PR92} and \cite[Ch.~4]{RS02} for necessary properties of P\l onka sums and to \cite{L98} for properties of inverse semigroups.) It is clear that $\wt{\mc{GP}}$ is contained in $\mc{GP} \circ \mcS$. Let $A = \sum_{i \in I}A_i$ be the P\l onka sum of groups $A_i$. Then each $A_i$ contains one idempotent, the identity $e_i$ of $A_i$. We will show that two elements $a, b \in A$ belong to one summand of $A$ precisely if $a a^{-1} = b b^{-1}$ and $a^{-1} a = b^{-1} b$. Note however that in each Clifford semigroup $x^{-1} x = x x^{-1}$. This product is the identity of the summand.
To show that $a^{-1} a = b^{-1} b$ for elements $a$ and $b$ in one summand of $A$, recall that $a, b \in A_i$ for some $i \in I$ precisely if $ab^{-1}b = a$ and $ba^{-1} a = b$. If these conditions hold, then
$a^{-1} a = a^{-1} a b^{-1}b = b^{-1} b a^{-1}a = b^{-1}b$, since Clifford semigroups satisfy the identity $x^{-1}x y^{-1}y = y^{-1}y x^{-1}x$. On the other hand, if $a^{-1} a = b^{-1} b$, then $ab^{-1}b = aa^{-1}a = a$ and $ba^{-1}a = bb^{-1}b = b$, since Clifford semigroups satisfy $xx^{-1}x = x$. Hence $a$ and $b$ belong to the same summand.

The elements of the form $a a^{-1}$ form all idempotents of the semigroup $A$. Note also the following identities true in all Clifford semigroups
\[
(x x^{-1}) \cdot (x x^{-1}) = x x^{-1} \,\, \mbox{and} \,\, (x x^{-1})^{-1} = x x^{-1}.
\]
It follows that $\sigma^p$ consists of the identities of the form
\[
r \cdot s = r,
\]
where $r = s$ are identities satisfied in $\wt{\mc{GP}}$ such that $r \cdot r = r$ and $r^{-1} = r$. Identities satisfied in $\wt{\mc{GP}}$ are described by P\l onka's Theorem (see \cite[Ch.~~4]{RS02} and \cite{BPR19}), whereas the last two conditions hold precisely if $r = r r^{-1} = s s^{-1} = s$. Hence
\[
\Sigma^p = \{r \cdot s = r \mid \wt{\mc{GP}} \models r = r r^{-1} = s s^{-1} = s \}.
\]
By Lemma~\ref{L:eqbase}, the Mal'tsev product $\mc{LZ} \circ \wt{\mc{GP}}$ satisfies the identities $\Sigma^p$.
\end{example}

\begin{example}
Consider the variety $\mc{MU}$ of monounary algebras $(A,u)$. Recall that all terms of monounary type have the form $u^m (x)$ for $m \in \mbN$. There are two types of proper subvarieties of $\mc{MU}$: the varieties $\mcU_k$, for $k \in \mbN$, each defined by the identity $u^{k}(x) = u^{k}(y)$, and the varieties $\mcU_{n,k}$, each defined by the identity $u^{n+k}(x) = u^{k}(x)$, for $k \in \mbN$ and $n \in \mbZ^{+}$. (See e.g. \cite[S.~9.1]{SP09}.) Note that an idempotent unary operation of an algebra $A$ is just the identity mapping on $A$. The only idempotent subvariety of $\mc{MU}$ is the variety $\mcU_{1,0}$ defined by the identity $u(x) = x$. A term idempotent $u^{m}(x)$ of a subvariety of $\mc{MU}$ must satisfy $u^{1+m}(x) = u^m(x)$. The only subvarieties of $\mc{MU}$ having term idempotents are the varieties $\mcU_k$ and $\mcU_{1,k}$ for $k \in \mbN$. Term idempotents have the form $u^{k+l}(x)$ for $k, l \in \mbN$. (See \cite{SP09} again.)
Now let $v = w$ be the unique identity defining a proper subvariety $\mcV$ of $\mc{MU}$. As the set $\Sigma$ we take the one element set consisting of this identity. Let $\mcW$ be any of the varieties $\mcU_k$ or $\mcU_{1,k}$. Then the identities of $\Sigma^p$ are obtained from $v = w$ by substituting for variables in $v = w$, $\mcW$-equivalent term idempotents of $\mcW$.
Finally note that the varieties $\mcU_{n,k}$ for $n > 1$ have no term idempotents, so algebras of these varieties have no idempotent elements, hence the set $\Sigma^p$ is empty. In this case the variety $\mf{H}(\mcV \circ \mcU_{n,k})$ generated by the Mal'tsev product $\mcV \circ \mcU_{n,k}$ coincides with $\mc{MU}$.
\end{example}

If $\mcU$ is the variety defined by $\Sigma^p$ for any equational base $\Sigma$ of $\mcV$, then
\[
\mcW \subseteq \mcV \circ \mcW \subseteq \mcU.
\]
By basic properties of free algebras in quasivarieties (see  e.g. \cite[\S3.3]{RS02}) it is known that, if $\varrho$ is the $\mcW$-replica congruence of the free $\mcU$-algebra $X \mcU$ over a set $X$, then
\[
X \mcU/ \varrho \, \cong \, X \mcW,
\]
where $X \mcW$ is the free $\mcW$-algebra over $X$.

\begin{proposition}\label{P:free}
Let $\mcU$ be the variety defined by $\Sigma^p$, where $\Sigma$ is any equational base of $\mcV$. Let $X \mcU$ be the free $\mcU$-algebra over a set $X$. Then
\[
X \mcU \in \mcV \circ \mcW.
\]
\end{proposition}
\begin{proof}
We already know that the $\mcW$-replica $\varrho$ of $F := X \mcU$ is isomorphic to $X \mcW$. Let $a \in F$ be such that the congruence class $B := a/\varrho$ is a subalgebra of $F$. We will show that $B$ belongs to $\mcV$.

Let $\sigma$ be an identity
$x_1 \dots x_k u = x_1 \dots x_k v$ from $\Sigma$
and let $r_1, \dots, r_k \in B$. Since $B$ is a single $\varrho$-class, it follows that $r_i/\varrho = r_j/\varrho$. As $F/\varrho$ is free in $\mcW$, we get that $\mcW$ satisfies the identities $r_i = r_j$.  Similarly, since $B$ is a subalgebra of $F$, we conclude that
$(r_1 \dots r_1 \omega)/\varrho = r_1/\varrho$ for each $\omega \in \Om$, whence $\mcW$ satisfies the identities $r_1 \dots r_1 \omega = r_1$. Thus the identity $r_1 \dots r_k u = r_1 \dots r_k v$ belongs to $\sigma^p \subseteq \Sigma^p$. So it is satisfied in $\mcU$. In particular in the free $\mcU$-algebra $F$ we have the equality
\[
r_1 \dots r_k u = r_1 \dots r_k v.
\]
This shows that $B$ satisfies the identities of $\Sigma$, so it belongs to $\mcV$.
\end{proof}

\begin{theorem}\label{T:thevar}
Let $\mcV$ and $\mcW$ be varieties of the same similarity type and let $\Sigma$ be an equational base for $\mcV$. The variety $\mf{H}(\mcV \circ \mcW)$ generated by the Mal'tsev product $\mcV \circ \mcW$ is defined by the identities $\Sigma^p$.
\end{theorem}
\begin{proof}
Let $\mcU$ be the variety defined by $\Sigma^p$. By Corollary \ref{C:VWU}, we know that $\mf{H}(\mcV \circ \mcW) \subseteq \mcU$, and by Proposition~\ref{P:free}, that $X \mcU \in \mcV \circ \mcW \subseteq \mf{H}(\mcV \circ \mcW)$.
Thus the varieties $\mf{H}(\mcV \circ \mcW)$ and $\mcU$ share the same free algebras. Since free algebras are uniquely defined and determine the variety, it follows that $\mf{H}(\mcV \circ \mcW) = \mcU$.
\end{proof}

\begin{corollary}
If not all $\mcW$-algebras contain idempotent elements, then $\mf{H}(\mcV \circ \mcW)$ does not depend on $\mcV$, and
coincides with the variety $\mcT$ of all $\Om$-algebras.
\end{corollary}

\begin{corollary}
If $\Sigma$ is an equational base for $\mcV$ and the Mal'tsev product $\mcV \circ \mcW$ is a variety, then $\Sigma^p$ is an equational base for $\mcV \circ \mcW$.
\end{corollary}

Let us recall that if $\mcW$ is an idempotent variety, then all blocks of congruences of algebras in $\mcV \circ \mcW$ are subalgebras, and the last condition in the definition of $\sigma^{p}$ is redundant.
Note, as well, that if $\mcW$ is idempotent, then each algebra $A$ in $\mcV \circ \mcW$ decomposes as a disjoint union of subalgebras belonging to $\mcV$  over the $\mcW$-replica of $A$. Following the convention adopted in \cite{BPR19} we may call such algebras \emph{$\mcW$-sums of $\mcV$-algebras}.

For the special case of $\mcW = \mcS_{\tau}$, where $\mcS_{\tau}$ is the variety equivalent to the variety $\mcS$ of semilattices,
one obtains Corollary 3.3 of \cite{BPR19}. Recall that two terms are $\mcS_{\tau}$-equivalent if and only if they contain exactly the same variables.

The following example, provided by C. Bergman, shows that, in general, one cannot replace the set $\mr{Id}(\mcW)$ in the definition of the identities true in $\mcV \circ \mcW$, by an equational base for $\mcW$.

\begin{example}
Consider the case where $\mcV = \mcW = \mc{LZ}$, the variety of left-zero bands.
It is defined by the unique identity $x \cdot y = x$. Every identity of $\mc{LZ}$ is of the form $u = v$, where $u$ and $v$ are groupoid words with the same first variable. By Theorem~\ref{T:thevar}, an equational base for
$\mf{H}(\mc{LZ} \circ \mc{LZ})$
consists of all identities of the form $u \cdot v = u$ again with $u$ and $v$ having the same first variable. This is the set $\Sigma^p$ of Definition \ref{D:eqbasprol}, if we take for $\Sigma$ the unique identity $x \cdot y = x$ playing the role of $\sigma$. However, if in the definition of $\sigma^{p}$, we replace $\mr{Id}(\mc{LZ})$ by the unique identity defining $\mc{LZ}$, then we will only get the identities:
\begin{align*}
&x \cdot x = x,\\
&(x \cdot y) \cdot x = x \cdot y,\\
&x \cdot (x \cdot y) = x,\\
&(x \cdot y) \cdot (x \cdot y) = x \cdot y.
\end{align*}
They are not sufficient to derive all the identities of $\mf{H}(\mc{LZ} \circ \mc{LZ})$.
\end{example}

\section{A sufficient condition for $\mcV \circ \mcW$ to be a variety}\label{S:Main}

Each algebra $A$ of the same similarity type as algebras of a variety $\mcW$ has a replica $A/\varrho$ in $\mcW$. The replica congruence $\varrho$ is the intersection of all congruences $\theta$ of $A$ with the quotient $A/\theta$ in $\mcW$. Below we provide another description useful in the proof of the main theorem of this section.

First, we define a binary relation $\varrho^0$ on $A$ as follows:
\begin{align*}
&(a, b) \in \varrho^0 \, \mbox{if and only if there are}\\ &\mbox{an identity} \, x_1\dots x_n p = x_1 \dots x_n q  \mbox{ true in } \mcW \, \mbox{and}\\ &d_1,\dots, d_n \in A  \mbox{ such that }  a = d_1 \dots d_n p,\, b = d_1\dots d_n q.
\end{align*}

\begin{lemma}\label{L:transclos}
Let $\alpha$ be a reflexive and symmetric binary relation on an algebra $A$ preserving the operations of $A$. Then the transitive closure $\tr(\alpha)$ of $\alpha$ is a congruence relation.
\end{lemma}
\begin{proof}
The relation $\tr(\alpha)$ is obviously an equivalence relation. So it remains to show that it preserves operations. Let $\omega \in \Omega$ be an $n$-ary operation and let $a_1, \dots, a_n, b_1, \dots, b_n \in A$ be such that $(a_i,b_i) \in  \tr(\alpha)$ for $1 \leq i \leq n$. Then for each $i$ there are elements
$d_{i}^1, \dots, d_{i}^{k_i}$, such that
\[
a_i\ =\ d_{i}^{1} \ \alpha \ d_{i}^{2} \ \alpha \ \dots \ \alpha \ d_{i}^{k_{i - 1}} \ \alpha \ d_{i}^{k_i} \ = \ b_i.
\]
Using reflexivity of $\alpha$ we can extend the sequences $d_{i}^{1}, \dots, d_{i}^{k_{i}}$ to sequences of the same length, say $k$, so that
\[
a_i\ =\ d_{i}^{1} \ \alpha \ d_{i}^{2} \ \alpha \ \dots \ \alpha  \ d_{i}^{k} \ = \ b_i.
\]
Since $\alpha$ preserves operations, we have
\[
a_1 \dots a_n \omega \ \alpha \ \dots \ \alpha \ d_{j}^{1} \dots d_{j}^{n} \omega \ \alpha \ d_{j+1}^{1} \dots d_{j+1}^{n} \omega \ \alpha \ \dots \ \alpha \ b_1 \dots b_n \omega.
\]
Thus $(a_1 \dots a_n \omega, b_1 \dots b_n \omega) \in \tr(\alpha)$.
\end{proof}

\begin{proposition}\label{L:Vreplicaform}
Let $\mcW$ be a variety and $A$ be an algebra of the same similarity type as algebras in $\mcW$. The $\mcW$-replica congruence $\varrho$ of $A$ coincides with the transitive closure $\tr(\varrho^0)$ of $\varrho^0$.
\end{proposition}
\begin{proof}
We will first prove that the relation $\tr(\varrho^0)$ is a congruence relation. It is obviously reflexive and symmetric.
We will show that it preserves the operations of $\Om$.
Let $\omega \in \Om$ be an operation of arity $n$ and let $a_1,\dots, a_n, b_1,\dots, b_n \in A$ be such that $a_i\ \varrho^0\ b_i$ for each $1\leq i\leq n$. In particular, this means that there are identities $x^i_1 \dots x^i_{k_i} p^i = x^i_1 \dots x^i_{k_i} q^i$ satisfied in $\mcW$, and elements $d^i_1, \dots, d^i_{k_i}\in A$ such that $a_i = d^i_1 \dots d^i_{k_i} p^i$ and $b_i = d^i_1 \dots d^i_{k_i} q^i$.
Then note that the identity
\[
(x^1_1 \dots x^1_{k_1} p^1) \dots (x^n_1 \dots x^n_{k_n} p^n)\, \omega = (x^1_1 \dots x^1_{k_1} q^1) \dots (x^n_1 \dots x^n_{k_n} q^n)\, \omega
\]
is also satisfied in $\mcW$. As $a_1 \dots a_n \omega = (d^1_1 \dots d^1_{k_i} p^1) \dots (d^n_1 \dots d^n_{k_n} p^n)\ \omega$ and $b_1 \dots b_n \omega = (d^1_1 \dots d^1_{k_1} q^1) \dots (d^n_1 \dots d^n_{k_n} q^n)\ \omega$, it follows that
\[
a_1 \dots a_n \omega  \ \varrho^0 \  b_1 \dots b_n \omega.
\]
Now by Lemma~\ref{L:transclos}, we conclude that $\tr(\varrho^0)$ is a congruence relation.

It remains to show that $A/\tr(\varrho^0)\in \mcW$, and that $\tr(\varrho^0)$ is the smallest congruence relation with this property. However, if $x_1 \dots x_n p = x_1 \dots x_n q$ is an identity true in $\mcW$, then
\begin{align*}
&\big(a_1/\tr(\varrho^0)\big) \dots \big(a_n/\tr(\varrho^0)\big)\, p =(a_1 \dots a_n p)/\tr(\varrho^0)\\
&=(a_1 \dots a_n q)/\tr(\varrho^0) = \big(a_1/\tr(\varrho^0)\big) \dots \big(a_n/\tr(\varrho^0)\big)\, q,
\end{align*}
since, by definition, $(a_1 \dots a_n p, a_1 \dots a_n q) \, \in \, \varrho^0 \subseteq \tr(\varrho^0)$. Thus $A/\tr(\varrho^0)$ satisfies every identity true in $\mcW$, and hence $A/\tr(\varrho^0)\in \mcW$.

On the other hand, if $\theta$ is a congruence of $A$ such that $A/\theta\in \mcW$, and $a_1, \dots ,a_n \in A$, then
$(a_1/\theta) \dots (a_n/\theta)\, p = (a_1/\theta) \dots (a_n/\theta)\, q$,
and hence $a_1 \dots a_n p\ \theta\ a_1 \dots a_n q$. Thus $\varrho^0 \subseteq \theta$, and we have $\tr(\varrho^0)\subseteq \tr(\theta) = \theta$.
\end{proof}

\begin{theorem}\label{T:main}
Let $\mcV$ and $\mcW$ be varieties of the same similarity type without nullary operation symbols, and let $\mcW$ be idempotent.
If there exist terms $f(x,y)$ and $g(x,y)$ such that
\begin{itemize}
\item[(a)] $\mcV \models f(x,y) = x\ \mbox{and}\ \mcV \models g(x,y) = y$,
\item[(b)] $\mcW \models f(x,y) = g(x,y)$,
\end{itemize}
then the Mal'tsev product $\mcV \circ \mcW$ is a variety.
\end{theorem}
\begin{proof}
Let $A$ be a quotient of an algebra belonging to the Mal'tsev product $\mcV \circ \mcW$, i.e. $A \in \mf{H}(\mcV \circ \mcW)$. Our aim is to prove that $A$ itself belongs to $\mcV \circ \mcW$.

Recall that the algebra $A$ has the $\mcW$-replica $A/\varrho$, and by Proposition~\ref{L:Vreplicaform}, $\varrho = \tr{\varrho^0}$.
First, we will show that the relation $\varrho^0$ is transitive, which implies that $\varrho = \varrho^0$.
Let $a, b, c \in A$, and $a \, \varrho^0 \, b \, \varrho^0 \, c$. This means that
there are identities $x_1 \dots x_n p_1 = x_1 \dots x_n q_1$  and  $y_1 \dots y_m p_2 = y_1 \dots y_m q_2$ satisfied in $\mcW$, and elements
$d_1, \dots, d_n, e_1, \dots, e_m \in A$ such that
\begin{align*}
&a = d_1 \dots d_n p_1,\ \ b = d_1 \dots d_n q_1, \\
&b = e_1 \dots e_m p_2,\ \ c = e_1 \dots e_m q_2.
\end{align*}
Now let us consider the identity
\begin{equation}\label{E:newid}
f(x_1 \dots x_n p_1, y_1 \dots y_m p_2) = g(x_1 \dots x_n q_1, y_1 \dots y_m q_2).
\end{equation}
It is clear that \eqref{E:newid} is satisfied in $\mcW$.
On the other hand, since $\mcW$ is idempotent, it follows by Theorem~\ref{T:thevar} that $A$ satisfies the identities
\begin{equation}\label{E:f1g1}
f(p_1,q_1) = p_1\,\, \mbox{and} \,\, g(p_2,q_2) = q_2.
\end{equation}
Thus
\begin{equation*}
\begin{split}
        f(d_1\dots d_n p_1, e_1\dots e_m p_2) &=f(d_1\dots d_n p_1, b)\\
        &=f(d_1\dots d_n p_1, d_1\dots d_n q_1)=d_1\dots d_n p_1=a,
    \end{split}
\end{equation*}
and similarly
\begin{equation*}
    \begin{split}
        g(d_1 \dots d_n q_1, e_1 \dots e_m q_2) &= g(b, e_1 \dots e_m q_2)\\
        &= g(e_1 \dots e_m p_2, e_1 \dots e_m q_2) = e_1 \dots e_m q_2=c.
    \end{split}
\end{equation*}
Hence, for the elements $d_1, \dots, d_n, e_1, \dots, e_m$, the left-hand side of the identity \eqref{E:newid} equals $a$ and its right-hand side equals $c$.
Consequently, $a\ \varrho^0\ c$, whence $\varrho^0$ is transitive, and $\varrho = \varrho^0$.

Next, we will show that each congruence class of $\varrho$ satisfies the identities true in $\mcV$. Since $\mcW$ is idempotent, by Lemma \ref{L:subcon}, all congruence classes of $\varrho$ are subalgebras of $A$. Consider the identities
\begin{equation}\label{E:efge}
f(x,y) = x\ \mbox{and}\ g(x,y) = y,
\end{equation}
true in the variety $\mcV$. If $x_1 \dots x_n p = x_1 \dots x_n q$ is satisfied in $\mcW$, then by Theorem~\ref{T:thevar}, $A$ satisfies the identities
\begin{equation}\label{E:efgepq}
f(p,q) = p\ \mbox{and} \ g(p,q) = q.
\end{equation}
Now let $a, b \in A$ and $a\ \varrho\ b$. Since $\varrho = \varrho^0$, it follows that $a = d_1 \dots d_n p$ and $b = d_1\dots d_n q$ for some identity $x_1 \dots x_n p = x_1 \dots x_n q$ satisfied in $\mcW$ and some $d_1, \dots, d_n\in A$. Then
\[
f(a,b) = f(d_1 \dots d_n p, d_1\dots d_n q) = d_1 \dots d_n p = a,
\]
and similarly,
\[
g(a,b) = g(d_1 \dots d_n p, d_1 \dots d_n q) = d_1 \dots d_n q = b.
\]
Therefore $B$ satisfies the identities \eqref{E:efge}.
We will use this fact to show that $B$ satisfies any identity $x_1 \dots x_k u = x_1 \dots x_k v$ true in $\mcV$.
Define terms $t_1, \dots, t_k$ in the following way:
\begin{align*}
&x_1 \dots x_k t_1 = f( f( \dots f(f(f(x_1,x_2),x_3),x_4) \dots ), x_{k-1}), x_k),\\
&x_1 \dots x_k t_2 = f( f( \dots f(f(g(x_1,x_2),x_3),x_4) \dots ), x_{k-1}), x_k),\\
&x_1 \dots x_k t_3 = f( f( \dots f(g(f(x_1,x_2),x_3),x_4) \dots ), x_{k-1}), x_k),\\
&\cdots\\
&x_1 \dots x_k t_k = g( f( \dots f(f(f(x_1,x_2),x_3),x_4) \dots ), x_{k-1}), x_k).
\end{align*}
Since $\mcW$ satisfies the identity $f(x,y)= g(x,y)$, it also satisfies the identities $t_i = t_j$, for every $1 \leq i,j \leq k$. By Theorem~\ref{T:thevar}, it follows that $A$ satisfies the identity
\[
t_1 \dots t_k u = t_1 \dots t_k v.
\]
Now let $b_1, \dots, b_k$ be elements of $B$.
Since the identities \eqref{E:efge} hold in $B$, it follows that $b_1 \dots b_k t_i = b_i$ for every $1 \leq i \leq k$. Thus,
\begin{equation*}
    \begin{split}
        b_1 \dots b_k u &= (b_1 \dots b_k t_1) \dots (b_1 \dots b_k t_k) u\\
        &= (b_1 \dots b_k t_1)\dots (b_1 \dots b_k t_k) v = b_1 \dots b_k v.
    \end{split}
\end{equation*}
Consequently, $B$ belongs to $\mcV$.
\end{proof}

\begin{corollary}
Let $\mcU$ be any variety of $\Om$-algebras. Let $\mcV$ and $\mcW$ be subvarieties of $\mcU$ satisfying the assumptions of Theorem~\ref{T:main}. Then the Mal'tsev product $\mcV \circ_{\mcU} \mcW$ is a variety.
\end{corollary}

The main result of \cite{BPR19} (Theorem~\ref{T:BPR}) follows easily as a corollary of Theorem~\ref{T:main}.

\begin{corollary}\cite{BPR19}\label{C:BPR}
Let $\mcV$ be a strongly irregular variety of a plural type $\tau$ and $\mcS_{\tau}$ be the variety of type $\tau$ that is equivalent to the variety of semilattices. Then $\mcV \circ \mcS_{\tau}$ is a variety.
\end{corollary}
\begin{proof}
Let $t(x,y) = x$ be a strongly irregular identity satisfied in $\mcV$. Define terms $f(x,y) = t(x,y)$ and $g(x,y) = t(y,x)$. Since $\mcS_{\tau}$ satisfies all regular identities of type $\tau$, it also satisfies $f(x,y) = g(x,y)$. As $\mcV$ satisfies $f(x,y) = x$ and $g(x,y) = y$, and the variety $\mcS_{\tau}$ is idempotent, it follows by Theorem~\ref{T:main} that $\mcV \circ \mcS_{\tau}$ is a variety.
\end{proof}

Let us note that the satisfaction of the three identities of Theorem~\ref{T:main} implies that the variety $\mcV \cap \mcW$ is trivial. That suggests the following conjecture: If $\mcV \cap \mcW$ is trivial, then the Mal'tsev product $\mcV \circ \mcW$ is a variety. The following example
shows that this conjecture is false. The example was originally provided by C. Bergman, however with a different proof.

\begin{example}
Let $\mcV$ be any variety of commutative groupoids, and let $\mc{LZ}$ be the variety of left-zero bands. Then the intersection $\mcV \cap \mc{LZ}$ is the trivial variety $\mc{E}$.

Let $A$ be the groupoid whose multiplication table is given below.

\begin{center}
    \begin{tabular}{c|cccc}
   $\cdot$ &0&1&2&3\\
    \hline
    0& 0 & 0 & 0 & 0\\
    1& 0 & 1 & 0 & 0\\
    2& 2 & 2 & 2 & 2\\
    3& 2 & 3 & 2 & 3\\
    \end{tabular}
\end{center}

The $\mc{LZ}$-replica congruence of $A$ is the congruence $\varrho$ with the congruence classes $\{0, 1\}$ and $\{2, 3\}$ which are semilattices. Hence $A$ is a member of $\mcV \circ \mc{LZ}$.

The groupoid $A$ has also a congruence $\theta$ with three congruence classes $\{0, 2\}, \{1\}, \{3\}$. The quotient $B = A/\theta$ has only one proper non-trivial congruence $\alpha$ with the classes $\{\{0,2\},\{3\}\}$ and $\{\{1\}\}$. The groupoid $B$ is neither commutative nor a left-zero band, and the quotient $B/\alpha$ is not a left-zero-band. Hence $B$ is not a member of $\mcV \circ \mc{LZ}$.
\end{example}

The following corollaries show that triviality of $\mcV \cap \mcW$ combined  with some types of congruence permutability is sufficient for the existence of the
two terms required by Theorem~\ref{T:main}. First recall that two congruences $\alpha$ and $\beta$ are $3$-\emph{permutable} (or $3$-\emph{permute}) if
\[
\alpha \circ \beta \circ \alpha = \beta \circ \alpha \circ \beta.
\]
An algebra is $3$-permutable if any two of its congruences $3$-permute, and a variety is $3$-permutable if all of its algebras are $3$-permutable.
Note as well the following.

\begin{lemma}\cite[Lemma~4.66]{MMT87}\label{L:joincong}
Let $\alpha$ and $\beta$ be congruences of an algebra $A$.
Then $\alpha \circ \beta \circ \alpha \supseteq \beta \circ \alpha \circ \beta$ if and only if $\alpha \vee \beta = \alpha \circ \beta \circ \alpha$.
\end{lemma}

We will also need one more lemma.

\begin{lemma}\label{L:3perm}
Let $\theta, \theta_1, \theta_2$ be congruences of an $\Om$-algebra $A$ such that $\theta \subseteq \theta_1, \theta_2$.
If the quotient $A/\theta$ is congruence $3$-permutable, then the congruences $\theta_1$ and $\theta_2$ are $3$-permutable.
\end{lemma}
\begin{proof}
Let $a, b \in A$ and $a\ \theta_1 \circ \theta_2 \circ \theta_1\ b$. Then
\[
a/\theta \quad (\theta_1/\theta) \circ (\theta_2/\theta) \circ (\theta_1/\theta) \quad b/\theta.
\]
By $3$-permutability of $A/\theta$, we also have
\[
a/\theta \quad (\theta_2/\theta) \circ (\theta_1/\theta) \circ (\theta_2/\theta) \quad b/\theta.
\]
So there are elements $c, d \in A$ such that
\[
a/\theta \quad (\theta_2/\theta) \quad c/\theta \quad (\theta_1/\theta) \quad d/\theta \quad (\theta_2/\theta) \quad b/\theta.
\]
Thus
$a\ \theta_2\ c\ \theta_1\ d\ \theta_2\ b$, whence finally
\[
\theta_1 \circ \theta_2 \circ \theta_1\subseteq \theta_2 \circ \theta_1 \circ \theta_2.
\]
The opposite inclusion is derived analogously.
\end{proof}

\begin{corollary}\label{C:cor1}
Let $\mcV$ and $\mcW$ be subvarieties of a variety $\mcU$ of $\Om$-algebras, and let $\mcW$ be idempotent.
If $\mcV \cap \mcW$ is the trivial variety and $\mcU$ is congruence $3$-permutable, then $\mcV \circ \mcW$ is a variety.
\end{corollary}
\begin{proof}
Consider the $\mcV$-replica congruence $\varrho_{\mcV}$, the $\mcW$-replica congruence $\varrho_{\mcW}$ and the
$\mcU$-replica congruence $\varrho_{\mcU}$ of the free algebra $X\Om$ for $X = \{x,y\}$.
Then $\varrho_{\mcU} \subseteq \varrho_{\mcV}, \varrho_{\mcW}$. By assumption $X\Om/\varrho_{\mcU}$ is congruence $3$-permutable.
Hence, by Lemma~\ref{L:3perm}, the congruences $\varrho_{\mcV}$ and $\varrho_{\mcW}$ are $3$-permutable.
Thus
\[
\varrho_{\mcV} \vee \varrho_{\mcW} = \varrho_{\mcV} \circ \varrho_{\mcW} \circ \varrho_{\mcV}.
\]

The quotient $X\Om/(\varrho_{\mcV} \vee \varrho_{\mcW})$ lies in $\mcV \cap \mcW$, so it is a trivial algebra. It follows that  all the elements of $X\Om$ form one congruence class. In particular,
\[
x \, \varrho_{\mcV} \circ \varrho_{\mcW} \circ \varrho_{\mcV}  \, y.
\]
This means that there are terms $f(x,y)$ and $g(x,y)$ such that
\[
x \,\, \varrho_{\mcV} \,\, f(x,y) \,\, \varrho_{\mcW} \,\, g(x,y) \,\, \varrho_{\mcV} \,\, y.
\]
By definition of the congruences $\varrho_{\mcV}$ and $\varrho_{\mcW}$, the variety $\mcW$ satisfies the identity $f(x,y) = g(x,y)$ and the variety $\mcV$ satisfies the identities $f(x,y) = x$ and $g(x,y) = y$. By Theorem~\ref{T:main}, $\mcV \circ \mcW$ is a variety.
\end{proof}

In the proof of the next corollary, we need Theorem~$7.3 (2)$ of \cite{FMcK17}. We state it below as a lemma.

\begin{lemma}\label{L:FMth}
Let $A$ be an algebra with a congruence $\theta$.
If $P(x,y,z)$ is a term which is Mal'tsev on each congruence class of $\theta$, then for any other congruence $\psi$ of $A$,
\[
\theta \circ \psi \circ \theta \, \subseteq \, \psi \circ \theta \circ \psi.
\]
\end{lemma}

\begin{corollary}\label{C:conpermid}
Let $\mcV$ be a congruence permutable variety and $\mcW$ be an idempotent variety. If $\mcV \cap \mcW$ is the trivial variety, then $\mcV \circ \mcW$ is a variety.
\end{corollary}
\begin{proof}
Let $\mcU = \mf{H}(\mcV \circ \mcW)$, and let
congruences $\varrho_{\mcV}$, $\varrho_{\mcW}$ and $\varrho_{\mcU}$ of the free algebra $X\Om$ for $X = \{x,y\}$ be defined the same way as in the proof of Corollary~\ref{C:cor1}. By Proposition~\ref{P:free} and Theorem~\ref{T:thevar}, it follows that the free $\mcU$-algebra $X\mcU$ on $X$ belongs to $\mcV \circ \mcW$.
Recall that $X\mcU$ is isomorphic to $X\Om/\varrho_{\mcU}$.
The $\mcW$-replica of $X\mcU$ is idempotent, so the congruence classes of the $\mcW$-replica congruence $\varrho_{\mcW}/\varrho_{\mcU}$ of $X\mcU$ are subalgebras of $X\mcU$ belonging to the congruence permutable variety $\mcV$. Thus there exists a term that is Mal'tsev on each congruence class of $\varrho_{\mcW}/\varrho_{\mcU}$. Then by Lemma~\ref{L:FMth},
\[
\varrho_{\mcW}/\varrho_{\mcU} \circ \varrho_{\mcV}/\varrho_{\mcU} \circ \varrho_{\mcW}/\varrho_{\mcU} \subseteq \varrho_{\mcV}/\varrho_{\mcU} \circ \varrho_{\mcW}/\varrho_{\mcU} \circ \varrho_{\mcV}/\varrho_{\mcU}.
\]
This implies that
\[
\varrho_{\mcW} \circ \varrho_{\mcV} \circ \varrho_{\mcW} \subseteq \varrho_{\mcV} \circ \varrho_{\mcW} \circ \varrho_{\mcV}.
\]
Therefore, by Lemma~\ref{L:joincong},
\[
\varrho_{\mcV} \vee \varrho_{\mcW} = \varrho_{\mcV} \circ \varrho_{\mcW} \circ \varrho_{\mcV}.
\]
The remaining part of the proof proceeds exactly as in the last part of the proof of Corollary~\ref{C:cor1}.
\end{proof}

\section{Band sums of algebras}\label{S:Bands}

In this section we provide some applications of Theorem~\ref{T:main}. We first consider $\Om$-algebras of a plural type $\tau$ with binary and possibly unary operations. Recall that a \emph{band} $(S,\cdot)$ is an idempotent semigroup. (For properties of bands see \cite{H76}.) It can be considered as an algebra of the same type $\tau$ by defining, for each $\omega \in \Om$, $xy \omega := x \cdot y$ in the case $\omega$ is binary, and $x \omega = x$ if $\omega$ is unary.

In all examples below, $f(x,y)$ and $g(x,y)$ denote the terms one needs in order to apply Theorem~\ref{T:main}.

\begin{example}\label{Ex:LB} \textbf{(Lattices and bands)}. Let $\mc{L}$ be any variety of lattices $(L, +, \cdot)$ and $\mc{B}$ any variety of bands considered as algebras $(B, +, \cdot)$ satisfying $x + y = x \cdot y$. Let
\[
f(x,y) = x + xy \,\, \mbox{and} \,\, g(x,y) = xy + y.
\]
Then the variety of all lattices satisfies $f(x,y) = x$ and $g(x,y) = y$, and the variety of all bands satisfies $f(x,y) = x \cdot xy = xy = xy \cdot y = g(x,y)$. By Theorem~\ref{T:main}, it follows that $\mc{L} \circ \mc{B}$ is a variety.

A small alteration shows that the same holds for the Mal'tsev product of the variety of Boolean algebras and a variety of bands. This time we consider Boolean algebras as algebras $(A, +, \cdot,')$ with $0$ defined by $x \cdot x' = y \cdot y'$ and $1$ defined by $x + x' = y + y'$, and bands as algebras of the same type satisfying $x + y = x \cdot y$ and $x' = x$. The same terms $f(x,y)$ and $g(x,y)$ as for lattices show that the assumptions of Theorem~\ref{T:main} are satisfied.
\end{example}

\begin{example}\label{Ex:QB} \textbf{(Quasigroups and bands)}. We consider quasigroups as algebras $(Q, \cdot, /, \setminus)$ with three binary operations, and bands as algebras of the same type with three equal binary band operations. Let
\[
f(x,y) = (x \cdot y) / y \,\, \mbox{and} \,\, g(x,y) = x \setminus (x \cdot y).
\]
Then the varieties of quasigroups satisfy $f(x,y) = x$ and $g(x,y) = y$ and the varieties of bands satisfy $f(x,y) = g(x,y)$. Theorem~\ref{T:main} shows that for any variety $\mc{Q}$ of quasigroups and any variety $\mc{B}$ of bands, the Mal'tsev product $\mc{Q} \circ \mc{B}$ is a variety.

Again a small alteration shows that the Mal'tsev product of any variety of loops and any variety of bands is a variety. Loops are considered here as algebras of the same type as quasigroups with the identity $1$ defined by $x / x = y \setminus y$.

As quasigroups and loops are congruence permutable, and the intersecion $\mc{Q} \cap \mc{B}$ is trivial, the statements of this example may also be proved using Corollary \ref{C:conpermid}.
\end{example}

\begin{example}\label{Ex:GB} \textbf{(Groups and bands)}. Groups are loops with associative multiplication. So by Example~\ref{Ex:QB} it follows that the Mal'tsev product of the variety of groups and the variety of bands is a variety. However, instead of algebras with three binary operations, groups may be considered as algebras $(G, \cdot, ^{-1})$ with one binary and one unary operation, where the identity $1$ is defined as $x \cdot x^{-1}$. (See e.g. \cite{BPR19}.) Bands are considered as algebras of the same type. As terms of Theorem~\ref{T:main} we take
\[
f(x,y) = x \cdot y^{-1} \cdot y \,\, \mbox{and} \,\, g(x,y) = x \cdot x^{-1} \cdot y.
\]
It is easy to see that these terms satisfy the assumptions of Theorem~\ref{T:main}.

A similar argument shows that if $\mcV$ is a variety of rings or modules, and bands are considered as algebras of the same type as $\mcV$, then $\mcV \circ \mc{B}$ is a variety.
\end{example}

Examples \ref{Ex:LB}--\ref{Ex:GB} above may be extended to varieties of any plural type. First we show that bands may be considered as algebras of any plural type $\tau$. Let $(A, \cdot)$ be any band. For a given (plural) type $\tau$, define operations of this type on $A$ as follows. For each $n$-ary $\omega \in \Om$, define the operation $\omega$ on $A$ as $x_1 \dots x_n \omega := x_1 \cdot \ldots \cdot x_n$. In this way, one obtains an algebra $(A,\Om)$ of type $\tau$. The semigroup operation $\cdot$ is recovered via $x \cdot y = xy\dots y \omega$ for any at least binary $\omega \in \Om$. It satisfies the associative law, and each at least binary operation $\omega$ may be obtained by composition from this binary operation. On the other hand, let $(A,\Om)$ be an idempotent algebra of a (plural) type $\tau$ such that for any two at least binary $\omega$ and  $\omega'$ in $\Om$, one has $x y \dots y \omega = x y \dots y \omega'$ and this binary operation is associative. Then by defining $x \cdot y := x y \dots y \omega$, one obtains a band $(A,\cdot)$. It is easy to see that the varieties of bands and $\Om$-bands are (definitionally) equivalent.
If $(A,\cdot)$ belongs to some special variety $\mc{B}$ of bands, the identities defining $\mc{B}$ may be ``translated'' into the language of $\Om$-algebras using the binary operation $\cdot$ derived from $\Om$-operations. In this way one obtains the variety of $\Om$-bands equivalent to the variety $\mc{B}$.

\begin{theorem}\label{T:bands}
Let $\mcV$ be a variety of some plural type $\tau$, and let $\mc{B}$ be a variety of $\Om$-bands of the same type.
Let $f(x,y)$ and $g(x,y)$ be terms containing both variables $x$ and $y$, and such that the first variable of both terms is the same and the last variable of both terms is the same.
If
\[
\mcV \models f(x,y) = x\ \mbox{and}\ \mcV \models g(x,y) = y,
\]
then $\mcV \circ \mc{B}$ is a variety.
\end{theorem}
\begin{proof}
First recall that the free band on two generators $x$ and $y$ consists of the following elements: $x, y, xy, yx, xyx, yxy$. Hence if $f$ and $g$ contain both variables $x$ and $y$ and have the same first and the same last variable, then the idempotence of bands implies that $f(x,y) = g(x,y)$. Thus the assumptions of Theorem~\ref{T:main} are satisfied.
\end{proof}

Let us note that algebras in varieties $\mcV \circ \mc{B}$ of Theorem~\ref{T:bands} are band sums of $\mcV$-algebras. In particular, these varieties of band sums contain as a subvariety the variety $\mcV \circ \mcS_{\tau}$ of semilattice sums of $\mcV$-algebras (provided that $\mc{B}$ is not trivial).

\end{document}